\documentclass[a4paper,11pt]{amsart}

\usepackage{amssymb,amscd}
\usepackage{amsmath,amsfonts,latexsym}
\usepackage{color}
\usepackage[latin1]{inputenc}
\usepackage{enumerate,afterpage}

\usepackage{mathrsfs}
\usepackage{mathtools}
\usepackage{stmaryrd}

\usepackage{amssymb,amscd}
\usepackage{amsmath,amsfonts,latexsym}
\usepackage{graphicx}
\usepackage{color}
\usepackage[english]{babel}
\usepackage[latin1]{inputenc}
\usepackage{enumerate,afterpage}
\usepackage[all]{xy}

\newcounter{notes}%

\newtheorem{cor}{Corollary}[section]

\newtheorem{prop}[cor]{Proposition}
\newtheorem{lemma}[cor]{Lemma}

\newtheorem{Theorem}[cor]{Theorem}

\newtheorem{Proposition}[cor]{Proposition}
\newtheorem{Corollary}[cor]{Corollary}

\newtheorem{introthm}{Theorem}

\usepackage{amsthm}

\makeatletter
\newtheorem*{rep@theorem}{\rep@title}
\newcommand{\newreptheorem}[2]{%
\newenvironment{rep#1}[1]{%
 \def\rep@title{#2 \ref{##1}}%
 \begin{rep@theorem}}%
 {\end{rep@theorem}}}
\makeatother

\newreptheorem{theorem}{Theorem}

\theoremstyle{definition}
\newtheorem{defi}[cor]{Definition}
 \newtheorem{Claim}[cor]{Claim}

\newtheorem{remark}[cor]{Remark}
\theoremstyle{plain}

\def\co{\colon\thinspace}

\newcommand{\cC}{{\mathcal C}}

\newcommand{\C}{{\mathbb C}}
\newcommand{\HH}{{\mathbb H}}

\newcommand{\N}{{\mathbb N}}

\newcommand{\AdS}{\mathbb{A}\mathrm{d}\mathbb{S}}

\newcommand{\PSL}{\mathsf{PSL}}

\newcommand{\CH}{\mathrm{CH}}

\newcommand{\RR}{\mathbb R}
\newcommand{\CC}{\mathbb C}
\newcommand{\CP}{\mathbb{CP}}

\newcommand{\RP}{\mathbb{RP}}

\newcommand{\Ein}{\mathrm{Ein}}

\let\oldtocsection=\tocsection

\let\oldtocsubsection=\tocsubsection

\let\oldtocsubsubsection=\tocsubsubsection

\renewcommand{\tocsection}[2]{\hspace{0em}\oldtocsection{#1}{#2}}
\renewcommand{\tocsubsection}[2]{\hspace{1em}\oldtocsubsection{#1}{#2}}
\renewcommand{\tocsubsubsection}[2]{\hspace{2em}\oldtocsubsubsection{#1}{#2}}

\AtBeginDocument{%
   \def\MR#1{}
}

\begin{document}

\title[Width]{Quasicircles and width of Jordan curves in $\CP^1$}

\author[F. Bonsante]{Francesco Bonsante}
\address{FB: Universit\`a degli Studi di Pavia}
\email{francesco.bonsante@unipv.it}
\urladdr{www-dimat.unipv.it/$\sim$bonsante}

\author[J. Danciger]{Jeffrey Danciger}
\address{JD: Department of Mathematics, University of Texas at Austin}
\email{jdanciger@math.utexas.edu}
\urladdr{www.ma.utexas.edu/users/jdanciger}

\author[S. Maloni]{Sara Maloni}
\address{SM: Department of Mathematics, University of Virginia}
\email{sm4cw@virginia.edu}
\urladdr{www.people.virginia.edu/$\sim$sm4cw}

\author[J.-M. Schlenker]{Jean-Marc Schlenker}
\address{JMS: Department of mathematics, University of Luxembourg}
\email{jean-marc.schlenker@uni.lu}
\urladdr{math.uni.lu/schlenker}

\thanks{Bonsante was partially supported by the Blue Sky research project ``Analytical and geometric properties of low dimensional manifolds''; Danciger was partially supported by an Alfred P. Sloan Foundation Fellowship and by NSF grants DMS-1510254 and DMS-1812216; Maloni was partially supported by NSF grant DMS-1506920, DMS-1650811, DMS-1839968 and DMS-1848346; Schlenker was partially supported by UL IRP grant NeoGeo and FNR grants INTER/ANR/15/11211745 and OPEN/16/11405402. The authors also acknowledge support from U.S. National Science Foundation grants DMS-1107452, 1107263, 1107367 ``RNMS: GEometric structures And Representation varieties'' (the GEAR Network).}

\date{\today}

\begin{abstract}
We study a notion of ``width" for Jordan curves in $\mathbb{CP}^1$, paying special attention to the class of quasicircles. The width of a Jordan curve is defined in terms of the geometry of its convex hull in hyperbolic three-space. A similar invariant in the setting of anti de Sitter geometry was used by Bonsante-Schlenker to characterize quasicircles amongst a larger class of Jordan curves in the boundary of anti de Sitter space. By contrast to the AdS setting, we show that there are Jordan curves of bounded width which fail to be quasicircles. However, we show that Jordan curves with small width are quasicircles.
\end{abstract}

\maketitle

\section{Results and motivations}\label{intro}

\subsection{The width of a Jordan curve in $\CP^1$}

Throughout we identify $\CP^1$ with the boundary at infinity of the hyperbolic three-space $\HH^3$.
Given a Jordan curve $C$ in $\CP^1$, let $\CH(C)$ denote the convex hull of $C$ in the 3-dimensional hyperbolic space $\HH^3$, namely the smallest closed convex set whose accumulation set at infinity is $C$. The boundary of $\CH(C)$ is the union of two properly embedded disks, denoted $\partial^+ \CH(C)$ and $\partial^- \CH(C)$. 

\begin{defi}\label{def:width}
The \emph{width} of a Jordan curve $C$ in $\CP^1$ is defined as:
\begin{align}\label{eqn:sup}
w(C) &= \sup_{x \in \CH(C)} \left(d(x,\partial^-\CH(C)) + d(x,\partial^+\CH(C))\right).
\end{align}
\end{defi}

Note that $w(C)$ may be infinite. 

A Jordan curve $C$ in $\CP^1$ is called a \emph{quasicircle} if it is the image of $\RP^1$ under a quasiconformal homeomorphism of $\CP^1$. Quasicircles arise, for example, as the limit sets of quasifuchsian surface groups. Such a quasicircle $C$ has $w(C) < \infty$. Indeed if $C$ is the limit set of the quasifuchsian group $\Gamma < \PSL(2,\CC)$, then the convex hull $\CH(C)$ is cocompact under the action of $\Gamma$, and hence the supremum in~\eqref{eqn:sup} is achieved at some point in $\CH(C)$.
In fact, it is true that any quasicircle has finite width. The main purpose of this article is to investigate to what extent the converse statement holds.

In Section \ref{width} we will prove the following result.
\begin{introthm} \label{example2}
There exist a Jordan curve $C$ with finite width which is not a quasicircle.
\end{introthm}

So the condition that the width is finite does not characterize quasicircles. However, Jordan curves with small width are quasicircles, as we will show in Section \ref{small}. The precise statement actually uses a slightly different notion of width, the ``boundary width'', defined as follows.

\begin{defi}\label{def:pwidth}
The \emph{boundary width} of a Jordan curve $C$ in $\CP^1$ is defined as:
\begin{align}\label{eqn:boundary}
w_\partial(C) &= \max\left(\sup_{x \in \partial^+\CH(C)} d(x,\partial^-\CH(C)), \sup_{x \in \partial^-\CH(C)}d(x,\partial^+\CH(C))\right).
\end{align}
\end{defi}

It follows from the definition that $w_\partial(C)\leq w(C)$, but the two quantities are different, see Section \ref{sc:ww}.

\begin{introthm} \label{tm:small}
  Let $w_0 = \mathrm{cosh}^{-1}(\sqrt{2})$. If $C$ is a Jordan curve in $\CP^1$ with $w_\partial(C) < w_0$, then $C$ is a quasicircle.
\end{introthm}

One key step in the proof of Theorem \ref{tm:small} is the following characterization of quasicircles in terms of a nearest point projection map from $\partial^+ \CH(C)$ to $\partial^- \CH(C)$.

\begin{introthm} \label{tm:proj}
  Let $C \subset \CP^1$ be a Jordan curve and let $\pi_+ \co \partial^+ \CH(C) \to \partial^- \CH(C)$ be a map sending each point $x \in \partial^+ \CH(C)$ to one of the (compactly many) nearest points on $\partial^- \CH(C)$. Then $C$ is a quasicircle if and only if $\pi_+$ is a quasi-isometry.
\end{introthm}

\subsection{Motivations from anti-de Sitter geometry}\label{AdS}

The main motivation for the investigations presented here can be found in analog, but somewhat simpler statements, that are known in anti-de Sitter geometry.

The $3$--dimensional anti-de Sitter (AdS) space $\AdS^3$ is the Lorentzian cousin of the $3$--dimensional hyperbolic space $\HH^3$. It is the model space for Lorentzian geometry of constant curvature $-1$ in dimension $2+1$. The projective boundary $\partial \AdS^3$ of $\AdS^3$ is a conformal Lorentzian space analogous to the Riemann sphere $\CP^1$ which is known as the Einstein space $\Ein^{1,1}$. The null lines on $\Ein^{1,1}$ determine two transverse foliations by circles which endow $\Ein^{1,1}$ with a product structure $\Ein^{1,1} \cong \RP^1 \times \RP^1$. 

Convex hull constructions in $\AdS^3$ are more subtle than in hyperbolic space because $\AdS^3$, differently from $\HH^3$, is not a convex space. In particular, an arbitrary collection of points in $\partial \AdS^3$ does not have a well-defined convex hull in $\AdS^3$. The Jordan curves $C$ in $\partial \AdS^3$ for which the convex hull $\CH(C)$ in $\AdS^3$ is well-defined are the  \emph{acausal meridians}, namely those Jordan curves arising as the graph of an orientation-preserving homeomorphism of $\RP^1$, and their limits (called achronal meridians). Amongst these, the natural analogue of quasicircles,  called here \emph{Einstein quasicircles} (as in \cite{convexhull}), are the graphs of orientation-preserving quasisymmetric homeomorphisms. 

In~\cite{maximal}, Bonsante-Schlenker define the \emph{width} $w_{\AdS}(C)$ of an acausal meridian~$C$ in terms of the timelike distances between points of the future boundary $\partial^+ \CH(C)$ of the convex hull and points of the past boundary $\partial^- \CH(C)$. Here is an equivalent definition (rewritten slightly to make the analogy with~\eqref{eqn:sup} transparent):
\begin{align}\label{eqn:width-AdS}
w_{\AdS}(C) = \sup_{x \in \mathrm{CH}(C)} \left\{d^T(x,\partial^-\mathrm{CH}(C)) + d^T(x,\partial^+\mathrm{CH}(C))\right\},
\end{align}
where $d^T(x, \partial^\pm \CH(C))$ denotes the maximum timelike distance between the point $x$ and any point $y$ in  $\partial^\pm \CH(C)$ which is causally related to $x$. Note that in anti-de Sitter geometry, we have that $w_{\AdS}(C)$ is also equal to $$w_{\AdS}(C) = \max\left(\sup_{x \in \partial^+\CH(C)} d^T(x,\partial^-\CH(C)), \sup_{x \in \partial^-\CH(C)}d^T(x,\partial^+\CH(C))\right)~, $$
as can be seen from the inverse triangle inequality for time-like triangles.
(As mentioned above, in the hyperbolic case the two definitions are different.) Note also that the width of an acausal meridian $C$ trivially satisfies $w_{\AdS}(C) \leq \pi/2$. In fact, Bonsante-Schlenker~\cite[Theorem 1.12]{maximal} characterize Einstein quasicircles as those for which the width is strictly less than the maximum possible. 

\begin{Proposition}[Bonsante--Schlenker]\label{pr:w_ads}
 An acausal meridian $C \subset \partial\AdS^3 = \Ein^{1,1}$ is an Einstein quasicircle if and only if $w_{\AdS}(C)<\frac{\pi}{2}$.
\end{Proposition}
 
Theorem~\ref{example2} shows that the naive analogue of Proposition~\ref{pr:w_ads} in hyperbolic geometry is false, in general.

\subsection{An analogy with minimal surfaces} 

It might be useful to point out an analogy between the results presented here and the relation between quasicircles and minimal (resp. maximal) surfaces in hyperbolic (resp. anti-de Sitter) geometry.
\begin{itemize}
\item Given a Jordan curve $\Gamma$ in $\partial\HH^3$, it always bounds a (possibly non-unique) minimal surface \cite{anderson:minimal2}. If this minimal surface has principal curvatures $|k_i|\leq 1-\epsilon<1$, then $\Gamma$ must be a quasicircle \cite{epstein:reflections}, but there are quasicircles that do not bound any minimal surface with curvature less than $1$. Seppi \cite{seppi:disks} recently proved that the principal curvatures of the minimal surface can be bounded from above by the quasisymmetric constant of the quasicircle, if it is small enough.
  \item Given an acausal curve $\Gamma\subset \partial \AdS^3$, it always bounds a maximal surface with principal curvatures at most $1$, and $\Gamma$ is a quasicircle if and only if it bounds a maximal surface with principal curvatures uniformly less than $1$ \cite{maximal}.
\end{itemize}
This analogy suggests natural questions, for instance whether a quasicircle in $\partial \HH^3$ with width less than an explicit constant (perhaps $w_0$) bounds a minimal surface with principal curvatures less than $1$.

\section{Width does not characterize quasicircles in $\CP^1$}\label{width}

\subsection{Quasicircles in $\partial \HH^3$} 

A Jordan curve $C$ in $\CP^1$ is called a $K$--\emph{quasicircle} if $C$ is the image of $\RP^1$ under a $K$--quasiconformal homeomorphism of $\CP^1$, see \cite{ahlfors}. Ahlfors gave a convenient characterization of hyperbolic quasicircles in terms of distance between points or `neck pinching'. 
\begin{Proposition}[Ahlfors \cite{ahlfors}]\label{alf}
  A planar Jordan curve $\Gamma\subset \C$ is a $K$--quasicircle if and only if it satisfies the $K$--{\em bounded turning condition}: there is a constant $K \geq 1$ such that for each pair of points $x, y \in \Gamma$ we have that 
  $$ \mathrm{diam}(\Gamma[x,y]) \leq K |x-y|,$$
where $\Gamma[x,y]$ is the subarc of $\Gamma$ joining $x$ and $y$ with smaller diameter.
\end{Proposition}

Another well-known statement that will be used below is the compactness of uniformly quasiconformal maps, see \cite[Theorems II.5.1 and II.5.3]{lehto-virtanen}, from which the following statement follows.

\begin{lemma} \label{lm:compactness}
Any sequence of uniform quasicircles has a subsequence converging either to a quasicircle or to a point. 
\end{lemma}

\subsection{A motivating example}\label{proof_example1}

In this section we prove the following result. This will motivate the construction we will use to prove Theorem \ref{example2}.

\begin{Theorem} \label{example1}
There exist $M>0$ and a sequence of $K_n$--quasicircles $C_n$ in $\C \mathbb{P}^1$ with $w(C_n) < M$ and such that $C_n$ converges in the Hausdorff topology to a limit $C$ which is neither a Jordan curve nor a point.
\end{Theorem}

The proof will give an explicit construction of such a sequence $C_n$. We will verify that each $C_n$ is a quasicircle using Ahlfors' criterion (Proposition~\ref{alf}).
Note that the quasicircle constants $K_n$ must tend to infinity -- indeed if $K_n$ were bounded, by standard compactness properties of quasicircles (Lemma \ref{lm:compactness}),
the limit of $C_n$ would be either a Jordan curve or a point.
The bounded width property will come from the following.

\begin{Proposition}\label{charac_bddwid}
  Let $C_n$ be a sequence of Jordan curves in $\C = \C \mathbb{P}^1\setminus \{\infty\}$ and let $\Omega_n^-$ and $\Omega_n^+$ be the ``external'' and ``internal'' complementary regions. Assume that for any sequence $g_n \in \PSL(2,\C)$ one of the following happens (up to passing to a subsequence):
  \begin{enumerate}
    \item either $g_n(C_n)$ converges to a point;
    \item or $g_n$ does not ``squeeze'' the complementary regions, in the sense that there are two open subsets $U_-$ and $U_+$ of $\C \mathbb{P}^1$ such that $U_- \subset g_n(\Omega_n^-)$ and $U_+ \subset g_n(\Omega_n^+)$ for all $n$.
  \end{enumerate}
  Then $w(C_n)< M$ for some $M$ independent of $n$. 
\end{Proposition}

 \begin{proof}
Suppose, by contradiction, that there exists a sequence of points $x_n \in \CH(C_n)$ such that $d(x_n, \partial^+ \CH(C_n)) +  d(x_n, \partial^- \CH(C_n))\to +\infty$. Let $g_n$ be any isometry sending $x_n$ to some fixed point $\bar x$ in $\HH^3$. First, we notice that no subsequence of $(g_n)_{n\in \N}$ can collapse the whole sequence of curves $(C_n)_{n\in \N}$ to a point, as this would contradict the fact that $x_n\in CH(C_n)$.  Second, since $(g_n)_{n\in \N}$ does not ``squeeze'' the complementary regions, let $U_+$ and $U_-$ be open sets as above and consider planes $P^+$ in $\HH^3$ with boundary in $U_+$ and $P^-$ in $\HH^3$ with boundary in $U_-$. Notice that $g_n(\partial^+ \CH(C_n))$ disconnects $P^+$ from $g_n(\partial^- \CH(C_n))$, and $g_n(\partial^- \CH(C_n))$ disconnects $P^-$ from $g_n(\partial^+ \CH(C_n))$, so that $d(x_n, \partial^+ \CH(C_n)) = d(\bar x, g_n(\partial^+ \CH(C_n))) < d(p, P^+)$ and $d(x_n, \partial^- \CH(C_n)) = d(\bar x, g_n(\partial^- \CH(C_n))) < d(\bar x, P^-)$. On the other hand, $d(\bar x, P^\pm)$ does not depend on $n$, so we have a uniform bound on $d(x_n, \partial^+ \CH(C_n)) +  d(x_n, \partial^- \CH(C_n))$ which contradicts the assumption.
\end{proof}

\subsubsection{Proof of Theorem \ref{example1}}

  The idea of this construction is to consider Jordan curves $C_n \subset \C$ which are piecewise unions of arcs of circles such that any two of these circles either meet forming a positive uniform angle or are uniformly disjoint, where by ``uniformly disjoint'' we mean that the modulus of the annulus bounded by them is uniformly bounded from both $0$ and $\infty$ (see  \cite[Section  I.6]{lehto-virtanen} for the definition and properties of  the modulus of a ring). We can then use the fact that the limit of the images of a circle through any family of isometries $g_n$ can be either a point or a circle and the fact that the ``transversality condition'' above prevents different circles from having the same limit. This will allow us to use Proposition \ref{charac_bddwid} and prove that their width is uniformly bounded. We will also see what are the ``necks'' of $C_n$ to consider in order to apply Proposition \ref{alf}.

\begin{figure}
[hbt] \centering
\includegraphics[height= 5 cm]{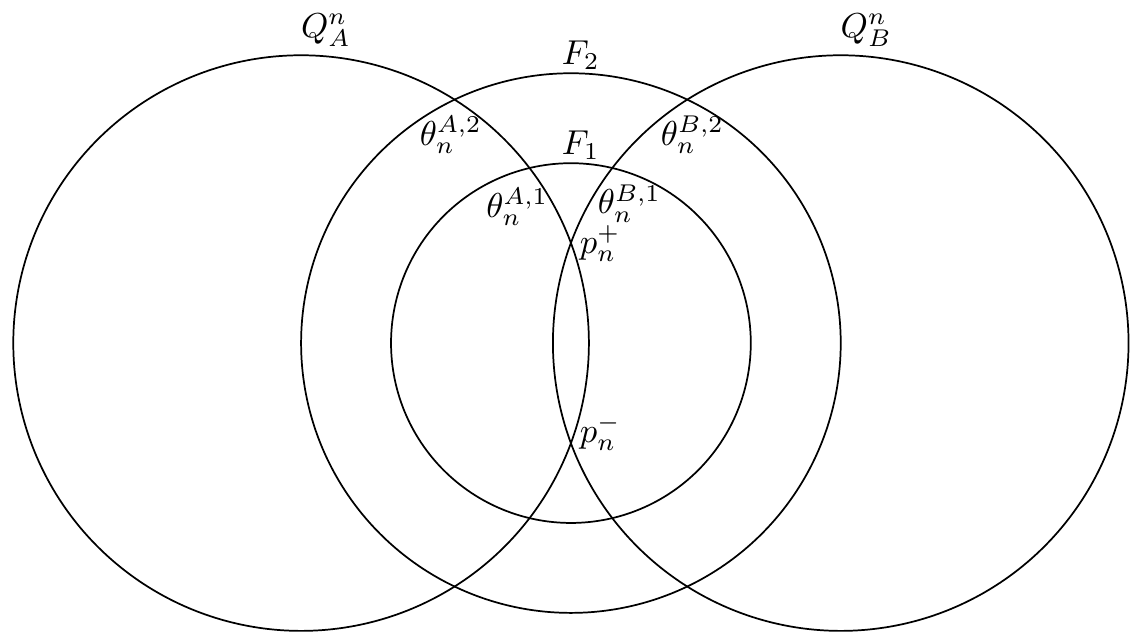}
\includegraphics[height= 5 cm]{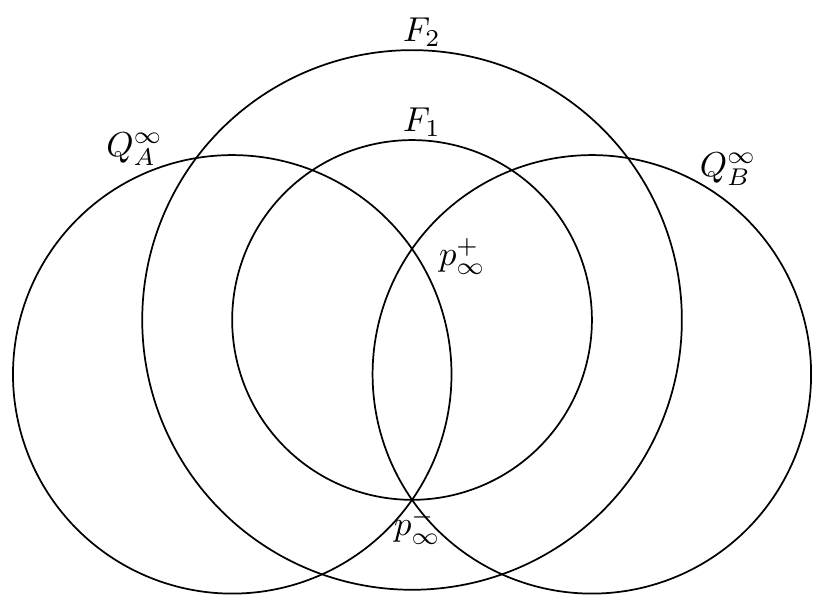}
\caption{The circles $F_1$ and $F_2$ together with the circles $Q_A^{n}$ and $Q_B^{n}$ (in the picture above) and together with the circles $Q_A^{\infty}$ and $Q_B^{\infty}$ (in the picture below).}
\label{fig:circles}
\end{figure}
   
 Fix two concentric circles $F_1$ and $F_2$  which bound disks $D_1$ and $D_2$ in the plane $\C = \C \mathbb{P}^1\setminus \{\infty\}$, so that $F_1 \subset \mathrm{Int}(D_2)$. Construct a sequence of pairs of circles $Q_A^n$ and $Q_B^n$ such that
\begin{itemize}
  \item $Q_A^n$ and $Q_B^n$ meet at points $p_-^n, p_+^n\in \mathrm{Int}(D_1)$ and form at these points an angle $\theta_n > 2\epsilon$ for some fixed $\epsilon>0$.
  \item $Q_A^n$ and $Q_B^n$ meet both $F_1$ and $F_2$ with some angle $\sphericalangle (Q_i^n, F_j) = \theta_n^{i,j} > 2\epsilon$ for $i\in \{ A,B\}, j\in \{1, 2\}$.
  \item $Q_A^n \to Q_A^{\infty}$ and $Q_B^n \to Q_B^{\infty}$, so that $p_-^n \to p_-^{\infty} \in Q_A^{\infty} \cap Q_B^{\infty} \cap F_1$.
\end{itemize}
See Figure \ref{fig:circles}.

\begin{figure}
[hbt] \centering
\includegraphics[height= 5 cm]{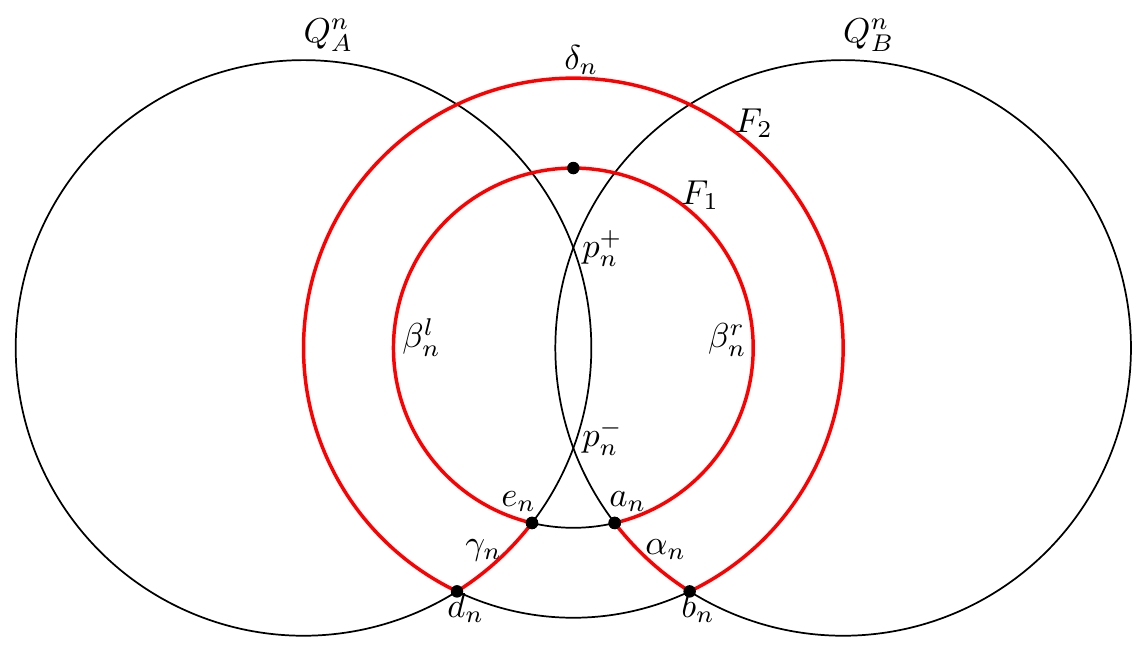}
\includegraphics[height= 5 cm]{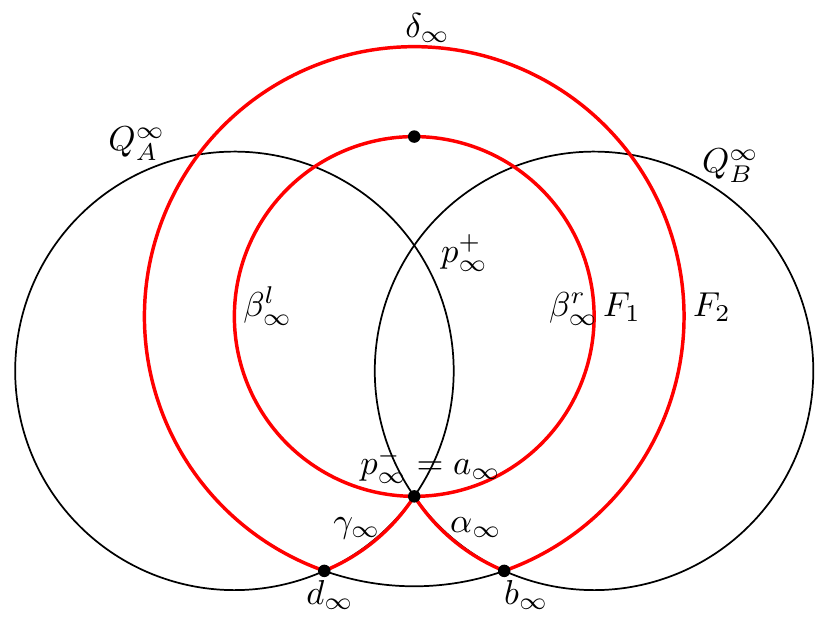}
\caption{The curve $C_n$ (above) and the curve $C_\infty$ (below) in red}
\label{fig:cn}
\end{figure}

For each $n$ consider the curve $C_n$ described in Figure \ref{fig:cn} and contained in the union $F_1 \cup F_2 \cup Q_A^n \cup Q_B^n$. Let $a_n$ be the vertex of $C_n$ at the intersection between $F_1$ and $Q_B^n$ and denote by $b_n,d_n,e_n$ the other vertices of $C_n$ ordered anti-clockwise. Let also $\alpha_n,\beta_n,\gamma_n, \delta_n$ be the circle arcs in $C_n$ named so that $\alpha_n$ joins $a_n$ to $b_n$, $\beta_n$ joins $b_n$ to $d_n$, $\gamma_n$ joins $d_n$ to $e_n$, and $\delta_n$ joins $e_n$ and $a_n$. We actually split the arc $\beta_n$ in two halves, $\beta_n^l$ and $\beta_n^r$, making sure that the limit arcs $\beta_\infty^l$ and $\beta_\infty^r$ are not degenerate. See Figure \ref{fig:cn}. 

Notice that $(a_n)_{n\in \N}$ and $(e_n)_{n\in \N}$ converges to the same vertex $a_\infty$ of $C_\infty$, while $(b_n)_{n\in \N}$ and $(d_n)_{n\in \N}$ converge respectively to vertices $b_\infty$ and $d_\infty$. The curve $C_\infty=\lim C_n$ is not a Jordan curve, so, by Lemma \ref{lm:compactness}, $C_n$ is not a sequence of uniform quasicircles, that is, there is no uniform $K> 0$ such that the $C_n$ are $K$--quasicircles. In fact, each $C_n$ is a $K_n$--quasicircle, but $\lim K_n=\infty$. (This can be seen directly by using Ahlfors' criterion and considering the necks defined by $(a_n)$ and $(e_n)$ and the diameter of $(\delta_n)$.)

To prove that the $C_n$ have uniformly bounded width, we will use the following simple definition and claim. 

\begin{figure}
[hbt] \centering
\includegraphics[height= 3.5 cm]{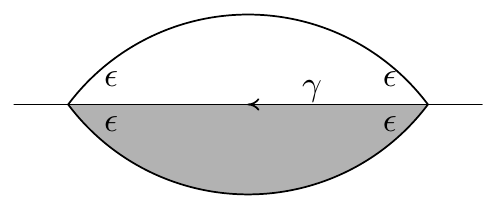}
\caption{The left (grey) and right (white) $\epsilon$--bigon of $\gamma$}
\label{bigon}
\end{figure}

\begin{defi} \label{df:theta}
  Let $\gamma\subset \CP^1$ be an oriented arc of circle, and let $\theta\in (0,\pi/2)$. The left (resp. right) $\theta$--bigon of $\gamma$ is the open bigon with angle $\theta$ on the left (resp. right) of $\gamma$. [Here by ``bigon'' we mean a domain of $\CP^1$ bounded by two arcs of circles. Those two arcs meet at the ``vertices'' of the bigon, and the interior angle at each vertex is the same.] See Figure \ref{bigon}.
\end{defi}

From Figure \ref{fig:cn} you can see the following claim.

\begin{Claim} \label{cl:theta}
  There exists $\epsilon>0$ (from the definition of $C_n$) such that for all $n\in \N$ and all oriented segments (of arcs of circle) $\gamma$ of $C_n$, the left and right $\epsilon$--bigons of $\gamma$ are disjoint from $C_n$ and contained in distinct regions of $\mathbb C\mathbb P^1\setminus C_n$.
\end{Claim}

Note that this claim only holds with the arc $\beta_n$ split as $\beta_n^l$ and $\beta_n^r$, as defined above.

Now, we claim that $w(C_n) \leq M$ for some $M$ independent of $n$. We will prove this by applying Proposition~\ref{charac_bddwid}. The following Claim \ref{claim} checks the hypotheses of Proposition~\ref{charac_bddwid}.

\begin{Claim}\label{claim}
  Let $(C_n)_{n\in \N}$ be the sequence of Jordan curves described above and let $\Omega_n^+$ and $\Omega_n^-$ be the ``external'' and ``internal'' complementary regions of $C_n$. Then for any sequence $(g_n)_{n\in \N}$ in $\PSL(2,\C)$, there exists a subsequence $(g_{n_k})_{k\in \N}$ such that either Condition $(1)$ or Condition $(2)$ from Proposition \ref{charac_bddwid} holds. 
\end{Claim} 

\begin{proof}
  Suppose first that (after taking a subsequence) $(g_n(\alpha_n))_{n\in \N}$, $(g_n(\beta_n))_{n\in \N}$, $(g_n(\gamma_n))_{n\in \N}$ and $(g_n(\delta_n))_{n\in \N}$ all converge to points. Then $(g_n(C_n))_{n\in \N}$ converges to a point.

  Otherwise, we can assume that, after taking a subsequence, one of the four sequences of segments, say $(g_n(\alpha_n))_{n\in \N}$, converges to an arc of circle $\alpha$ in $\CP^1$. Let $U_{n,l}$ and $U_{n,r}$ be the left and right $\epsilon$--bigons of $g_n(\alpha_n)$, with $\epsilon$ coming from the definition of $C_n$ and Claim \ref{cl:theta}. Then $U_{n,l}\to U_l$ and $U_{r,n}\to U_r$, where $U_l$ and $U_r$ are the left and right $\epsilon$--bigons of $\alpha$. We can then take $U_+, U_-$ to be $U_l, U_r$ respectively, and see that the second case in the Claim applies.
\end{proof}

\subsection{Proof of Theorem \ref{example2}} \label{proof_example2}

\begin{figure}
[hbt] \centering
\includegraphics[height= 4cm]{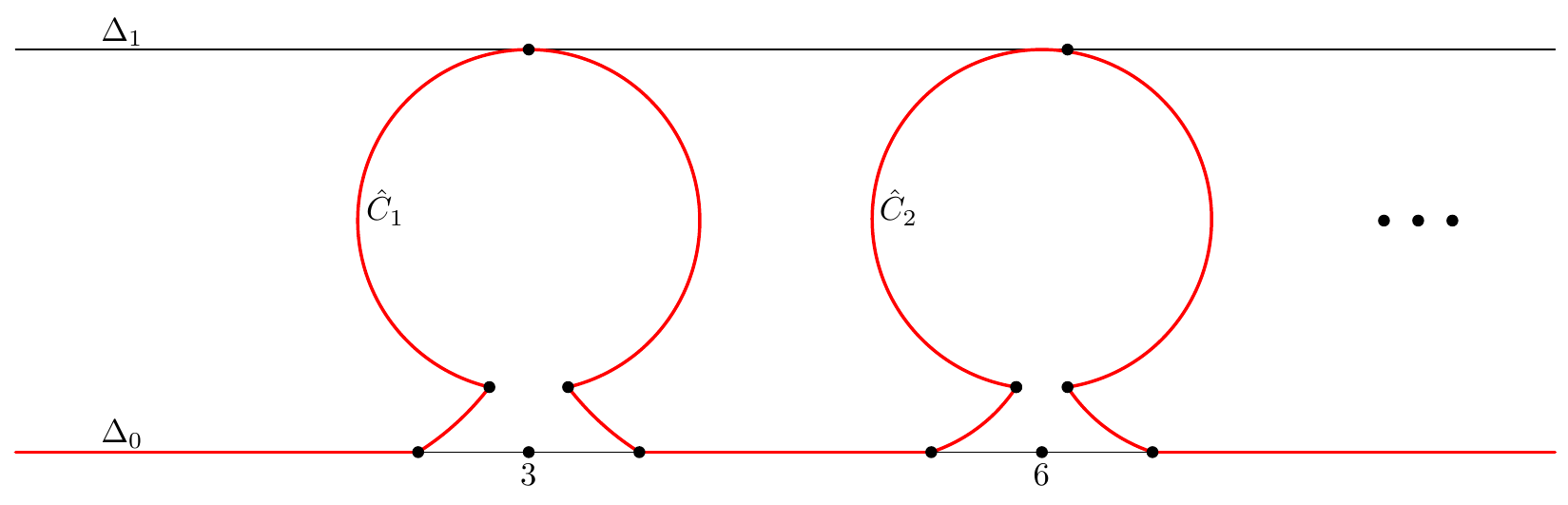}
\caption{The curve $D$ and the arcs $f_0, f_1, \cdots, f_7, \cdots$.}
\label{curve_D}
\end{figure}

In order to prove Theorem \ref{example2} we need to construct a Jordan curve $D$ with bounded width which is not a quasicircle. We will define $D$ as a curve containing the ``interesting'' part of all the quasicircles $C_n$ described in the previous section as follows.

We assume that the circle $F_2$ considered before has $\mathrm{diam}(F_2)=1$. We define a curve $D$ as follows. We start from the real axis $\Delta_0$ in $\C$, and for each $n\geq 1$ we remove a segment of $\Delta_0$ centered at $3n$ and glue instead a translated copy of $\hat{C}_n = C_n\setminus \delta_n$, scaled so that the highest point is on the line $\Delta_{1}$ of equation $\Im z =1$. We obtain in this manner a subset $D\subset \CP^1$, see Figure \ref{curve_D}. By construction, $D$ is a Jordan curve, but it is not a quasicircle. Indeed if $D$ were a $K$-quasicircle, then the  translates $D_n=D-3n$ would form a sequence of  $K$-quasicircles with uniform $K$. 
However their limit is not a Jordan curve, and this contradicts the compactness properties of $K$ quasicircles, Lemma \ref{lm:compactness}.

Denote by $f_0, f_1, \ldots, f_n, \ldots$ the arcs of circle composing $D$, with $f_0$ corresponding to the part of the real axis to the left of the first surgery, all oriented towards $+\infty$. Similarly to Claim \ref{cl:theta} and using the fact that the circles either meet forming a positive uniform angle or are uniformly disjoint, we note that for each $n\in \N$, the left and right $\epsilon$--bigons of $f_n$ are disjoint from $D$, for some $\epsilon>0$.
We will also consider the half-lines $\Delta_{0,+}$ and $\Delta_{1,+}$ composed of the points of $\Delta_0$ and $\Delta_1$, respectively, with positive real parts. It will be useful to note that the left $\epsilon$--bigon of $\Delta_{0,+}$ and the right $\epsilon$--bigon of $\Delta_{1,+}$ are disjoint from $D$, too, because all of $D$ is below $\Delta_0$ and above $\Delta_{1}$ by construction.

To complete the proof of Theorem A, we will prove that $D$ has finite width using Proposition \ref{charac_bddwid} and the following result. The proof is an extension of the idea in the proof of Claim \ref{claim}.

\begin{Proposition} \label{pr:D}
  Let $\Omega^+$ and $\Omega^-$ be the ``external'' and ``internal'' complementary regions of $D$. For any sequence $(g_n)_{n\in \N}$ in $\PSL(2,\C)$, there exists a subsequence $(g_{n_k})_{k\in \N}$ such that either Condition $(1)$ or Condition $(2)$ from Proposition \ref{charac_bddwid} holds.
\end{Proposition}

\begin{proof}
  We will use an auxiliary spherical metric $\rho$ on $\CP^1$, and consider two cases.

  First, if $\limsup_{n\to\infty}\sup_{m\in \N} \ell_\rho\left(g_n(f_m)\right)>0$, where $\ell_\rho$ is the length with respect to $\rho$, then there are strictly increasing sequences $(n_k)_{k\in \N}$ and $(m_k)_{k\in \N}$ such that $\left(g_{n_k} (f_{m_k})\right)_{k\in \N}$ converges in the Hausdorff topology to an arc of circle $a$. Then, as in the proof of Claim \ref{claim} above, any closed subsets of the left and right $\epsilon$--bigons of $a$ show that (2) in Proposition \ref{pr:D} holds.

  Second, suppose that $\limsup_{n\to\infty}\sup_{m\in \N} \ell_\rho(g_n(f_m))=0$. Then all the $f_m$ collapse to points. Since there are infinitely many segments of arcs of $D$ connecting $\Delta_{0,+}$ to $\Delta_{1,+}$, this implies that, after extracting a subsequence, $(g_n\Delta_{0,+})_{n\in \N}$ and $(g_n\Delta_{1,+})_{n\in \N}$ have the same limit $\lambda$, which can be either a point, an arc of circle, or a full circle. If $\lambda$ is a point, clearly all of $g_nD$ converges to $\lambda$, and (1) holds. If $\lambda$ is an arc of circle, then, since the left $\theta$-bigon of $\Delta_{0,+}$ and the right $\theta$-bigon of $\Delta_{1,+}$ are disjoint from $D$, case (2) of Proposition \ref{pr:D} holds. If $\lambda$ is a full circle, then $g_nD$ converges to $\lambda$ (again after extraction of a subsequence) and (2) holds again.
\end{proof}

\section{Jordan curves with small width (Proof of Theorem~\ref{tm:small})} \label{small}

We now consider the boundary width $w_\partial C$ of a Jordan curve, see Definition \ref{def:pwidth}. We have already noted that $w_\partial(C)\leq w(C)$. In this section, we prove the following strengthening of Theorem~\ref{tm:small}:

\begin{Theorem}\label{tm:B-improved}
Let $w_0 = \mathrm{cosh}^{-1}(\sqrt{2})$. There is a function $k:(0,w_0)\to(0,+\infty)$ such that if $C$ is a Jordan curve in $\CP^1$ with $w_\partial(C) \leq w < w_0$, then $C$ is a $k$--quasicircle where $k = k(w)$. 
\end{Theorem}

Theorem~\ref{tm:B-improved}, and hence Theorem~\ref{tm:small}, follows from Lemma~\ref{lm:pi-}, which we prove in Section~\ref{sec:3.1}, and from Theorem~\ref{tm:proj}, which we prove in Section~\ref{proofC}.

\subsection{Consequences of small width}\label{sec:3.1}

\begin{lemma} \label{lm:PP'Q}
  Let $P,P', Q$ be three planes in $\HH^3$ bounding disjoint closed half-spaces. Let $x\in P, x'\in P', y\in Q$. Then $d(x,y)\geq w_0$ or $d(x',y)\geq w_0$. 
\end{lemma}

\begin{proof}
  Let $y\in Q$ minimize the sum of the distance to $x$ and to $x'$, that is
  $$d(x,y) + d(x',y) = \mathrm{min}_{z \in Q} (d(x, z) + d(x', z)).$$ 
  Let $\Pi$ be the plane containing $x,x',y$. Clearly $\Pi$ contains the geodesics from $x$ to $y$ and from $x'$ to $y$, and the lines $\Pi\cap P$, $\Pi\cap P'$ and $\Pi\cap Q$ bound disjoint half-planes of $\Pi$. It is sufficient to prove the analog statement in the hyperbolic plane $\Pi$. In that case, the worst case is obtained when $\Pi\cap P$, $\Pi\cap P'$ and $\Pi\cap Q$ are pairwise asymptotic lines forming a triangle $\Delta$ and $y$ is in the position so that the symmetry in the line orthogonal to $\Pi\cap Q$ at $y$ exchanges $\Pi\cap P$ and $\Pi\cap P'$, while $x\in \Pi\cap P$ and $x'\in \Pi\cap P'$ minimize the distance to $y$ in the corresponding lines, see Figures \ref{PP'Q} and \ref{triang}. The segments $\overline{xy}$ and $\overline{x'y}$ decompose $\Delta$ in four hyperbolic triangles with one right angle, one ideal vertex, and opposite edge of the same length (by symmetry). So these four triangles are all congruent and have angles $0$, $\pi/2$ and $\pi/4$.
\begin{figure}
[hbt] \centering
\includegraphics[height= 8cm]{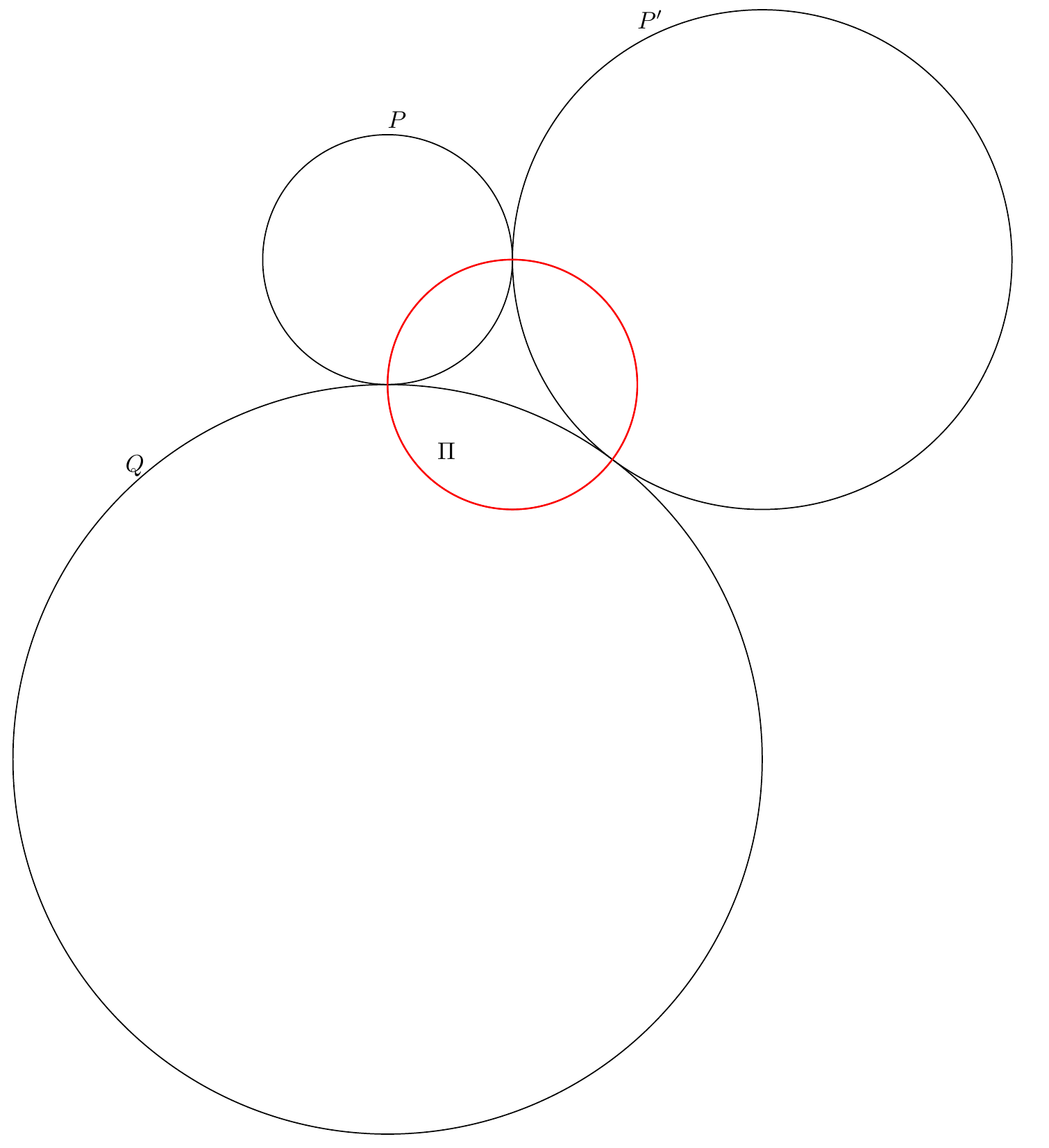}
\caption{The planes $P$, $P'$, $Q$ and $\Pi$ (in red), as in Lemma \ref{lm:PP'Q}.}
\label{PP'Q}
\end{figure}

  By the cosine formula for hyperbolic triangles:
  $$ \cos(0)=\cos(\pi/2)\cos(\pi/4)+\cosh(d(x,y))\sin(\pi/2)\sin(\pi/4)~, $$
  and therefore $d(x,y)=\cosh^{-1}(\sqrt{2})$ as claimed.
\end{proof}

\begin{figure}
[hbt] \centering
\includegraphics[height= 5cm]{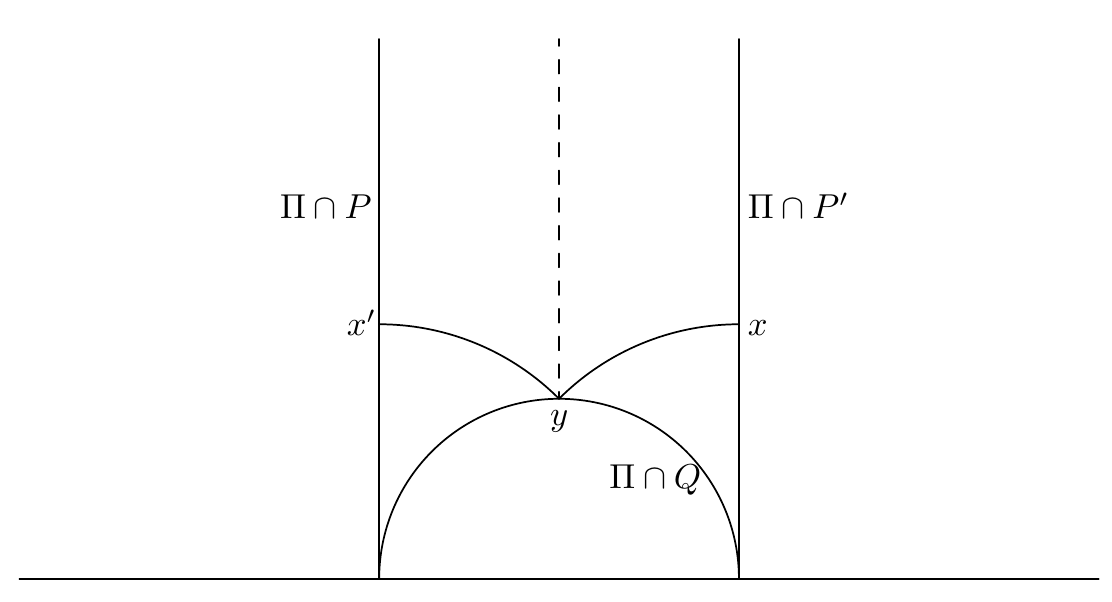}
\caption{The triangle $\Delta$ defined by the lines $\Pi\cap P$, $\Pi\cap P'$ and $\Pi\cap Q$ in the plane $\Pi$.}
\label{triang}
\end{figure}

\begin{lemma}\label{lem:roof}
For every $w<w_0$ there exists $b = b(w) > 0$ satisfying the following: Whenever $P, P'$ and $Q$ are three oriented planes such that $P$ (resp. $P'$) and $Q$ bound disjoint half-spaces and $x\in P$, $x'\in P'$ and $y\in Q$ are points such that $d(x,y)\leq w$ and $d(x',y)\leq w$, then there exists a point $z\in P\cap P'$ such that $d(x, z) + d(x', z) \leq b$.
\end{lemma}

\begin{proof}
For any $P \ni x$, $P' \ni x'$, and $Q \ni y$ as in the lemma statement, $P \cap P'$ is non-empty by Lemma~\ref{lm:PP'Q}.
The lemma follows by compactness of the space of such pointed triples of planes.
\end{proof}

Now, given a Jordan curve $C$ in $\CP^1$, denote by $\cC=\mathrm{CH}(C)$. Let $\pi_+\co\partial_-\cC\to \partial_+\cC$ be a nearest point projection map, meaning for each $y \in \partial_- \cC$, $\pi_+(y) \in \partial_+ \cC$ realizes the distance $d(y, \partial_+ \cC)$. Note that $\pi_+$ is not uniquely defined by this property, since the collection of points in $\partial_+ \cC$ nearest to $y \in \partial_- \cC$ may, in general, be a non-singleton compact set. In particular, we do not assume $\pi_+$ is continuous. Alternatively, one could consider a uniquely defined ``coarse" projection map which is set-valued, but we choose not to do this. Similarly, let $\pi_- \co \partial_+ \cC \to \partial_- \cC$ be a nearest point projection map in the opposite direction. Let $d_{\pm}$ be the induced metric on $\partial_{\pm}\cC$. 

\begin{Corollary} \label{lm:b}
For every $w < w_0$ there exists constants $a,b > 0$ such that whenever $w_\partial(C) \leq w$, we have:
\begin{enumerate}
\item If $y,y' \in \partial_- \cC$ satisfy $d(y,y') \leq a$, then $d_+(\pi_+(y), \pi_+(y')) \leq b$.
\item If $x, x' \in \partial_+ \cC$ satisfy $d(x,x') \leq a$, then $d_-(\pi_-(x), \pi_-(x')) \leq b$.
\end{enumerate}
\end{Corollary}

\begin{proof}
Set $w'$ so that $w < w' < w_0$ and let $b = b(w') > 0$ be the constant from Lemma~\ref{lem:roof}. Define $a = w' - w$.  Assume $w_\partial(C) \leq w$. We prove the first statement. The second is similar.

Let $y,y' \in \partial_- \cC$ be such that $d(y,y') \leq a$, and let $x = \pi_+(y)$ and $x' = \pi_+(y')$. Let $P,P',Q,Q'$ be support planes to $\cC$ at $x,x',y,y'$ respectively. By definition of boundary width, $d(x,y), d(x',y') \leq w_\partial(C)\leq w $.  Hence $d(x,y) \leq w'$ and $d(x',y) \leq w + a = w'$.
Since $P$ and $Q$ bound disjoint half-spaces as do $P'$ and $Q$, Lemma~\ref{lem:roof} gives a point $z \in P \cap P'$ so that $d(x,z) + d(x',z) \leq b$. Hence there is a path along $P \cup P'$ from $x$ to $x'$ of distance $\leq b$. The projection of this path onto $\partial_+ \cC$ has less or equal length (projection onto a convex set is contracting). Hence $d_+(x,x')\leq b$ as desired.
\end{proof}

\begin{remark}
Similarly  as in the proof of Corollary~\ref{lm:b}, Lemma~\ref{lem:roof} also implies that for each $w < w_0$, there exists $b = b(w) > 0$, so that whenever $w(C) \leq w$ the following holds: if $y \in \partial_- \cC$ and $x,x' \in \partial_+ \cC$ realize the minimum distance $d(y, \partial_+ \cC)$, then $d(x,x') \leq b$. In other words, the different possibly choices of a nearest point projection map $\pi_+$ are all within a uniform distance $b$ of one another, provided the width is smaller than $w$. 
\end{remark}

In the final lemma of this section, we show that the maps $\pi_+$ and $\pi_-$ are quasi-inverse quasi-isometries between $\partial_+ \cC$ and $\partial_- \cC$ whenever the width is small enough. 

\begin{lemma} \label{lm:pi-}
 Assume $w_\partial(C)\leq w< w_0$. Then the closest-point projection maps $\pi_+, \pi_-$ are quasi-inverse quasi-isometries with constants depending on $w$.
\end{lemma}

\begin{proof}
Let $a,b > 0$ be the constants from Corollary~\ref{lm:b}. Let $y, y' \in \partial_- \cC$, and let $x = \pi_+(y), x' = \pi_+(y')$. By subdividing the geodesic in $\partial_- \cC$ from $y$ to $y'$ into $N = \lceil\frac{d_-(y, y')}{a}\rceil$ arcs of length $\leq a$, and applying Corollary~\ref{lm:b} $N$ times, we obtain that 
$$d_+(x,x') \leq Nb \leq (b/a)d_-(y,y') + b.$$
Similarly for any $x,x' \in \partial_+ \cC$ and
$y = \pi_-(x)$ and $y' = \pi_-(x')$,
 $$d_-(y,y') \leq (b/a)d_+(x,x') + b.$$
Further, if $x = \pi_+(y)$ and $y' = \pi_-(x)$, 
 then it follows from Lemma~\ref{lem:roof} that $d_-(y,y') \leq b$. Hence $\pi_- \circ \pi_+$ is bounded distance from the identity map. It follows that $\pi_+$ and $\pi_-$ are quasi-inverse quasi-isometries with constants depending only on $a$ and $b$, which in turn depend only on $w$.
\end{proof}

\subsection{Proof of Theorem \ref{tm:proj}} \label{proofC}

We reformulate Theorem \ref{tm:proj} below as Proposition \ref{prop:proj-hyp}. This characterization of quasicircles in hyperbolic geometry is an analog of a result obtained in the AdS setting in \cite{convexhull}.

\begin{prop} \label{prop:proj-hyp}
Let $\cC \subset \HH^3$ be the convex hull of a Jordan curve $C \subset \CP^1$, its two boundary components denoted $\partial^+ \cC$ and $\partial^- \cC$. Consider a map $\pi_+\co \partial^- \cC \to \partial^+ \cC$ sending a point $x \in \partial^+ \cC$ to the (or to one of the compactly many) nearest point(s) on $\partial^- \cC$. Then $C$ is a quasicircle if and only if $\pi_+$ is a quasi-isometry. Further the quasi-isometry constants are bounded in terms of the quasicircle constant, and conversely.
\end{prop}

In the proof we will use the following result.

\begin{lemma}\label{lem:keyy}
Suppose $f\co \HH^3 \to \HH^3$ is an $L$-bilipschitz diffeomorphism. Let $\overline{f}$ be the extension of $f$ to $\HH^3 \cup \CP^1$, let $C = \overline{f}(\RP^1)$, let $\cC = \CH(C) \subset \HH^3$, and let $\partial^+ \cC, \partial^- \cC$ be the two boundary components of $\cC$ in $\HH^3$. Then there are constants $\eta_0, L', A'$ depending only on $L$ so that the width satisfies $w(\cC) \leq \eta_0$ and the path metrics on $\partial^+ \cC$ and on $\partial^- \cC$ are $(L', A')$ quasi-isometrically embedded. 
\end{lemma}

\begin{proof}
  Let $\delta = \delta_{\HH^2}$ be the $\delta$-hyperbolicity constant for the hyperbolic plane $\HH^2$. If $x \in \partial^+\cC$, then $x$ lies in the convex hull of three points of $C$, and hence is at distance at most $\delta$ away from a geodesic $\Delta$ of $\HH^3$ contained in $\partial^+ \cC$. Since $f^{-1}$ is an $L$-bilipschitz diffeomorphism, $f^{-1}(\Delta)$ is a smooth quasigeodesic with endpoints in $\RP^1$. By the Morse Lemma, $f^{-1}(\Delta)$ lies in a $D$-neighborhood of the geodesic $\Delta'$ in $\HH^3$ with the same endpoints, where $D>0$ depends only on $L$. Let $\mathcal H \subset \HH^3$ denote the totally geodesic hyperbolic plane bounded by $\RP^1$.
  Since the endpoints of $\Delta'$ are contained in $\RP^1$, $\Delta'$ is contained in $\mathcal H$ and hence all points of  $f^{-1}(\Delta)$ are within distance at most $D$ from $\mathcal H$. We conclude that for any point $x\in \partial^+\cC$, $f^{-1}(x)$ is at distance at most $L\delta+D$ from $\mathcal H$.
  
  The orthogonal projection of $f^{-1}(\partial^+\cC)$ on $\mathcal H$ is surjective, since $\mathcal H$ is totally geodesic and $\partial_\infty(f^{-1}(\partial^+\cC))=\overline{f}^{-1}(C) = \RP^1$. It follows that for all $y\in \mathcal H$, $y$ is at distance at most $L\delta+D$ from $f^{-1}(\partial^+\cC)$. The same arguments shows that $y$ is also at distance at most $L\delta+D$ from $f^{-1}(\partial^-\cC)$, and that for all $x\in \partial^-\cC$, $f^{-1}(x)$ is at distance at most $L\delta+D$ from $\mathcal H$.

  Next, let $z\in \cC$. Let $\Delta_0$ be the geodesic orthogonal to $\mathcal H$ containing $f^{-1}(z)$. The extreme points of $\Delta_0\cap f^{-1}(\cC)$ are points of $\Delta_0\cap f^{-1}(\partial \cC)$, so $f^{-1}(z)$ is contained in an interval $J$ of $\Delta_0$ bounded by $\Delta_0 \cap \mathcal H$ and a point of either  $f^{-1}(\partial^+\cC)$ or  $f^{-1}(\partial^-\cC)$, and we suppose without loss in generality it is the former (the other case is handled in the same manner). It follows from the previous argument that the length of the interval $J$ is less than $D + L \delta$ and hence $d(f^{-1}(z),f^{-1}(\partial^+\cC))\leq D+L\delta$. We also have that $d(f^{-1}(z),\mathcal H)\leq D+L\delta$ and therefore $d(f^{-1}(z),f^{-1}(\partial^-\cC))\leq 2D+2L\delta$.
 As a consequence, using again that $f$ is $L$-Lipschitz, $d(z,\partial^+\cC)\leq LD+L^2\delta$, and $d(z,\partial^-\cC)\leq 2LD+2L^2\delta$. Since this holds for all $z\in \cC$, we obtain that $w(C)\leq 3LD+3L^2\delta =: \eta_0$.
 
Next, consider the foliation of $\HH^3$ by surfaces $\Sigma_r$ at constant signed distance $r$ from $\mathcal H$.
We have already shown that for $|r| > L \delta + D$, the surface $\Sigma_r$ is disjoint from $f^{-1}(\cC)$, and hence $f(\Sigma_r)$ is disjoint from $\cC$. We choose the sign convention for $r$ so that when $r > L \delta + D$, the surface $f(\Sigma_r)$ lies on the concave side of $\partial^+ \cC$.
Fix some $r > L \delta + D$. Note that points of $f(\Sigma_r)$ lie within distance $L(r + D + L\delta)$ of $\partial^+ \cC$. This follows because a point of $\Sigma_r$ is distance $r$ from $\mathcal H$ and any point of $\mathcal H$ lies within distance $D + L \delta$ of $f^{-1}(\partial^+ \cC)$, as argued above, hence points of $\Sigma_r$ are within distance $r + D +L \delta$ of $f^{-1}(\partial^+ \cC)$ and that bound gets worse at most by a factor of $L$ when applying $f$.

Since $df$ stretches and compresses tangent vectors by at most a factor of $L$, it follows that the path metric on $f(\Sigma_r)\subset \HH^3$ is $L$-bilipschitz to the path metric on $\Sigma_r \subset \HH^3$ which itself is a $\cosh(r)$--bilipschitz embedded copy of $\HH^2$ in $\HH^3$. 

Consider two points $x,y \in \partial^+ \cC$, and let $x',y' \in f(\Sigma_r)$ be points within distance $L(r + D + L\delta)$ from $x,y$ respectively. Write $x' = f(a'), y' = f(b')$ and let $[a',b']_{\Sigma_r}$ be a geodesic in the path metric on $\Sigma_r$; its length is equal to $d_{\mathcal H}(a',b') \cosh r$, where $d_{\mathcal H}(\cdot, \cdot)$ denotes the distance after projection to $\mathcal H$. The length of $f([a',b']_{\Sigma_r})$ in the path metric of $f(\Sigma_r)$ is at most $L \cosh(r)d_{\mathcal H}(a',b') \leq L \cosh(r)d(a',b')$, where $d(\cdot, \cdot)$ denotes distance in $\HH^3$. Then, letting $d_+(\cdot,\cdot)$ denote the induced path metric on $\partial^+ \cC$, we have the following, where the first inequality comes from the fact that projections onto convex surfaces are distance decreasing:
\begin{align*}
d_+(x,y) &\leq d(x,x') + d(y,y') + \mathrm{length}(f([a',b']_{\Sigma_r})\\
& \leq 2L(r + D + L \delta) + L \cosh(r)  d(a',b') \\
& \leq 2L(r + D + L\delta) + L^2 \cosh(r) d(x',y')\\
& \leq 2L(r + D + L \delta) + L^2 \cosh(r) \left(d(x,y) + d(x',x) + d(y',y) \right)\\
&\leq 2L(r + D + L \delta) + L^2 \cosh(r) \left( d(x,y) + 2L(r + D + L \delta) \right)\\
& = L^2\cosh(r) d(x,y) +  2L(r + D + L \delta) (\cosh(r) L^2 +1).
\end{align*}
Hence the induced metric on $\partial^+ \cC$ is $(L',A')$ quasi-isometrically embedded for $L' = L^2 \cosh(r)$ and $A' = 2 L(r + D + L \delta) (\cosh(r) L^2 +1)$. The same argument proves $\partial^- \cC$ is also $(L', A')$ quasi-isometrically embedded (for the same $L', A'$).
\end{proof}

We are now ready to prove Proposition \ref{prop:proj-hyp}. 

\begin{proof}[Proof of Proposition \ref{prop:proj-hyp}]
Suppose $C$ is $k$-quasicircle. Since any quasiconformal map of $\CP^1$ extends to a bilipschitz map of $\HH^3$ (with constant depending on $k$), see Tukia--V\"{a}is\"{a}l\"{a} \cite[Theorem 3.11]{tukia_vaisala}, Lemma~\ref{lem:keyy} shows that $w(C) \leq \eta_0$ and $\partial^\pm \cC$ are $(L,A)$-quasi-isometrically embedded, where $\eta_0,L,A$ depend only on $k$. Consider $x,y \in \partial^-\cC$. Then $d(x, \pi_+(x)), d(y,\pi_+(y)) \leq \eta_0$. As a consequence
\begin{align*}
d_+(\pi_+(x), \pi_+(y)) &\leq Ld(\pi_+(x), \pi_+(y)) + A\\
&\leq L ( d(x,y) + 2\eta_0) + A\\
&\leq L( d_{\partial^- \cC}(x,y) + 2\eta_0) + A,
\end{align*}
and
\begin{align*}
d_+(\pi_+(x), \pi_+(y)) &\geq  d(\pi_+(x), \pi_+(y))\\
\geq d(x,y) - 2\eta_0
&\geq \dfrac{1}{L}d_{-}(x,y)-A - 2\eta_0,
\end{align*}
so $\pi_+$ is a quasi-isometric embedinng, with constants depending only on the quasiconformal regularity of $C$. Similarly, $\pi_-$ is a quasi-isometric embedding, with constants depending only on the quasiconformal regularity of $C$. Since $\pi_- \circ \pi_+$ and $\pi_+ \circ \pi_-$ are each at most $2 \eta_0$ away from the identity map, $\pi_+$ and $\pi_-$ are quasi-inverses, hence quasi-isometries. 

Conversely, suppose that $\pi_+\co\partial^-\cC\to \partial^+\cC$ is a $(L,A)$-quasi-isometry. Let $\Omega^+$ and $\Omega^-$ be the connected component of $\CP^1\setminus C$ facing $\partial^+\cC$ and $\partial^-\cC$, respectively. According to a theorem of Sullivan \cite{sullivan_travaux,epstein-marden}, there is a constant $K>1$ and $K$-bilipschitz maps $b^\pm\co\partial^\pm \cC\to \Omega^\pm$ (where $\Omega^\pm$ is equiped with the hyperbolic metric in its conformal class), with $b^\pm$ extending continuously to the identity on $C$. So the composition $b^+\circ\pi_+\circ (b^-)^{-1}\co\Omega^-\to\Omega^+$ is $(K^2L,K^2A)$-quasi-isometric and extends continuously to the identity on $C$, because the same is true for $b^+, b^-$ and indeed $\pi_+$ as well, since $\pi_+$ moves points at most by $\eta_0$.

Let now $u_\pm\co\Omega^\pm\to \HH^2$ be uniformization maps. Then the composition $u_+\circ b^+\circ\pi_+\circ (b^-)^{-1}\circ (u_-)^{-1}\co\HH^2\to\HH^2$ is a $(K^2L,K^2A)$-quasi-isometry, so its boundary extension is quasi-symmetric, with a quasi-symmetric norm depending only on $(L,A)$. It is therefore the boundary extension of a $k$-quasi-conformal map $q:\HH^2\to \HH^2$, with $k$ depending only on $(L,A)$. Since $b^+\circ\pi_+\circ (b^-)^{-1}$ extends continuously to the identity on $C$, we deduce that $q$ extends to the map $u_+\circ (u_-)^{-1}$ over $\partial\HH^2$ (where we are implicitly using Caratheodory's Theorem which ensures that the uniformization maps $u_\pm$ extend over the circle). As a consequence, the composition $(u_+)^{-1}\circ q\circ u_-\co\Omega_-\to\Omega_+$ is a $k$-quasi-conformal map that extends the identity over the boundary.

It now follows using standard arguments of \cite{ahlfors-reflections} that $C$ is the image of a circle in $\CP^1$ by a quasiconformal deformation, with quasiconformal factor depending only on the constants $(L,A)$.
\end{proof}

\subsection{Optimality of $w_0$} 

A natural question is: is the value for $w_0$ in Theorem \ref{tm:small} optimal? The following example shows that $w_0 = \mathrm{sinh}^{-1}(\sqrt{2})$ will not work. Note also that since $\mathrm{cosh}^{-1}(x) = \ln(1+\sqrt{x^2-1})$ and $\mathrm{sinh}^{-1}(x) = \ln(1+\sqrt{x^2+1})$, then the optimal value for $w_0$ is in the interval $[0.88137, 1.14622]$.
\begin{figure}
[hbt] \centering
\includegraphics[height= 4cm]{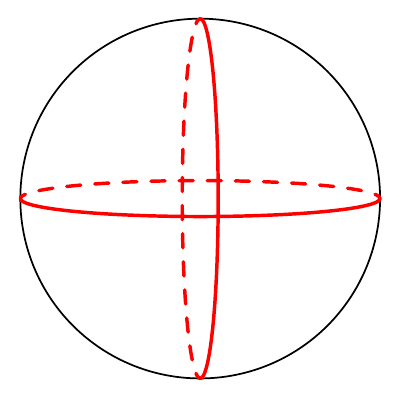}
\caption{The curve $G$ in $\CP^1$.}
\label{curve_G}
\end{figure}

\begin{figure}
[hbt] \centering
\includegraphics[height= 12cm]{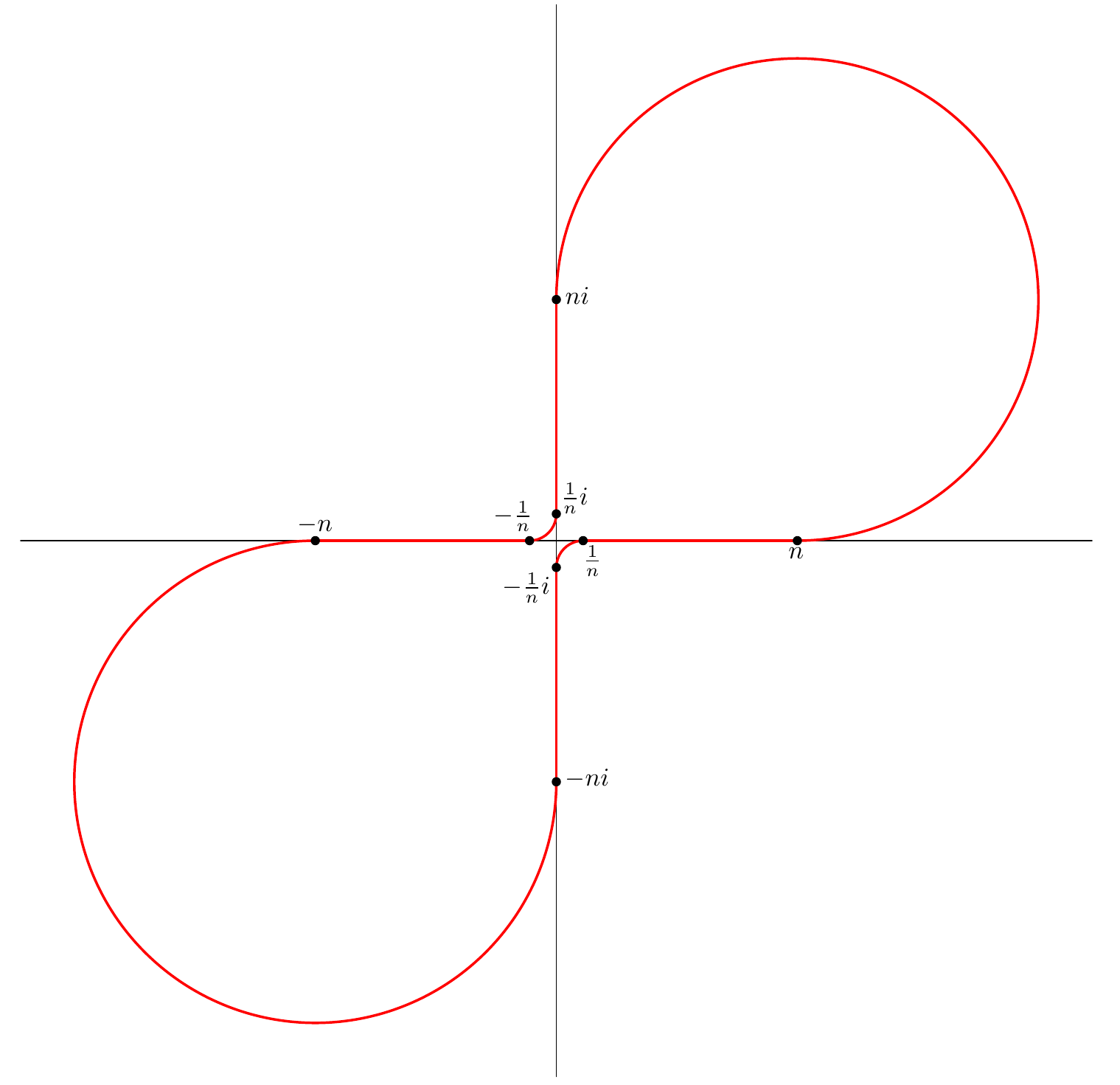}
\caption{The curve $G_n$ in $\C$.}
\label{curve_Gn}
\end{figure}

Let $G_n$ be the Jordan curves defined as follows. Start with the curve $G$ defined as the union of the two axes $\Re z = 0$ and $\Im z = 0$ in the plane $\C$, see Figure \ref{curve_G}. From $G$ remove the set $$\{z \in G \mid |\Re(z)| \in [0,\frac{1}{n})\cup(n, \infty), |\mathrm{Im}(z)| \in [0,\frac{1}{n})\cup(n, \infty)\}$$
and add arcs of circles, as shown in Figure \ref{curve_Gn}. The curves $G_n$ limit to the curve $G$, which is not a quasicircle.

\begin{figure}
[hbt] \centering
\includegraphics[height= 5cm]{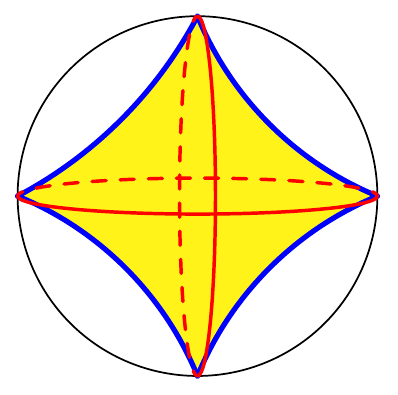}
\caption{Convex hull in $\HH^3$ of the curve $G$ in $\CP^1$.}
\label{hull_curve_Gn}
\end{figure}

Recall from Definition \ref{def:pwidth} that the boundary width $w_\partial(C)$ of a Jordan curve $C$ is the supremum over points of $\partial CH(C)$ of the distance to the other boundary component.

\begin{Proposition}
 $\lim_{n\to\infty} w_\partial(G_n) = \mathrm{sinh}^{-1}(\sqrt{2}).$
\end{Proposition}

\begin{proof}
  First, note that the convex hulls $\mathrm{CH}(G_n)$ are nested and limit to $\mathrm{CH}(G)$. We can then see that the limit $L := \lim_{n\to\infty} w_\partial(G_n)$ can be calculated as the boundary width of the limit curve $G$, where to make sense of the definition of $w_\partial(G)$, we decompose the boundary $\partial \mathrm{CH}(G)$ (which has four connected components) into two pieces, $\partial_+\mathrm{CH}(G)$ and  $\partial_-\mathrm{CH}(G)$, by taking the limits of $\partial_+\mathrm{CH}(G_n)$ and  $\partial_-\mathrm{CH}(G_n)$.  Looking at all the symmetries of this picture, we can see that $L$ corresponds to the maximum distance between one point on one face of $\mathrm{CH}(G)$ and the union of the two adjacent faces, see Figure \ref{hull_curve_Gn}.
  
  This calculation is easy because the picture has a lot of symmetries. The maximum will be achieved on any hyperbolic plane $\Pi$ meeting the geodesic $0 \infty$ orthogonally.  The calculation reduces to calculating the distance in the ideal hyperbolic quadrilateral $\mathcal{Q}$ in the hyperbolic plane $\Pi$ between a point on one of the sides and the union of the two adjacent sides. Using again the symmetry of this picture, we need to calculate the distance between the middle point of one of the sides to one of the two opposite sides. If we consider the Minkowski model in the usual coordinates, we can assume the vertices of $\mathcal{Q}$ to correspond to $[(1,1,0)]$, $[(1,0,1)]$, $[(1,-1,0)]$ and $[(1,0,-1)]$. The middle point of the side between $[(1,1,0)]$ and $[(1,0,1)]$ corresponds to $m = [(\sqrt{2},\frac{1}{\sqrt{2}},\frac{1}{\sqrt{2}})]$, while the geodesic between $[(1,1,0)]$ and $[(1,0,-1)]$ corresponds to the line $\ell$ defined as the subspace orthogonal to $v = [(1,1,-1)]$. Using a formula for the distance between a point and a line in the Minkowski plane, we can see that $\mathrm{sinh}(d(m, \ell)) = | \langle m, v \rangle | = \sqrt{2}$, as we wanted to prove. 
\end{proof}

\section{Comparing the width and the boundary width}
\label{sc:ww}

In this section, we show that the boundary width can indeed be strictly smaller than the width.

\begin{Proposition}
There exists a sequence of quasicircles $\Lambda_n \subset \CP^1$ so that $w_\partial(\Lambda_n)$ is uniformly bounded, but $w(\Lambda_n) \to \infty$.
\end{Proposition}

\begin{proof}
The quasicircles $\Lambda_n$ will be the limit sets of quasifuchsian groups $\Gamma_n$ whose geometric limit develops a rank two cusp. We recall the construction of Kerckhoff-Thurston~\cite{KT90}.

  Let $S$ be a closed oriented surface of genus $2$, let $\gamma$ be a simple closed curve on $S$.
  Fix two conformal metrics $c_-, c_+$ on $S$. By the Ahlfors-Bers Theorem \cite{ber_sim, ahl_the}, there exists a unique geometrically finite hyperbolic manifold $M$ which is homeomorphic to $(S \times \RR) \setminus (\gamma \times \{0\})$ and so that the conformal metric on $S \times \{\infty\}$ is $c_+$ and the conformal metric on $S \times \{-\infty\}$ is $c_-$. The end of $M$ associated to $\gamma \times \{0\}$ is a rank two cusp.
 
 Holding the conformal structures $c_+$ and $c_-$ at $S \times \{\infty\}$ and $S \times \{-\infty\}$ fixed, and performing hyperbolic Dehn filling on the cusp with filling slope $(1,n)$ yields hyperbolic manifolds $M_n$ which converge to $M$ in the Gromov-Hausdorff sense. Each manifold $M_n$ is quasifuchsian, and in particular homeomorphic to $S \times \RR$. Let $\Gamma_n < \PSL_2 \CC$ be the holonomy group of $M_n$ and let $\Lambda_n \subset \CP^1$ be the limit set of $\Gamma_n$, a quasicircle.
 
 The convex hull $\CH(\Lambda_n)$ covers the convex core $\mathscr{C}_n$ of $M_n$.  Note $\mathscr{C}_n$ is compact. The two boundary components $\partial^+ \CH(\Lambda_n)$ and $\partial^- \CH(\Lambda_n)$ cover the two boundary surfaces $\partial^+ \mathscr{C}_n$ and $\partial^- \mathscr{C}_n$ of $\mathscr{C}_n$, each homeomorphic to $S$.  In the limit as $n \to \infty$, the convex cores $\mathscr{C}_n$ converge to the convex core $\mathscr{C}$ of $M$, which is no longer compact since it contains the cusp. However, $\partial^+ \mathscr{C}_n$ and $\partial^- \mathscr{C}_n$ converge respectively to two compact surfaces $\partial^+ \mathscr{C}$ and $\partial^- \mathscr{C}$ bounding $\mathscr{C}$. By compactness, the maximum distance in $\mathscr{C}$ from a point on $\partial \mathscr{C}^-$ (resp. $\partial \mathscr{C}^+$) to $\partial \mathscr{C}^+$ (resp. $\partial \mathscr{C}^-$) is finite. Hence the maximum distance from a point of $\partial^+ \mathscr{C}_n$ (resp. $\partial^- \mathscr{C}_n$) to $\partial^- \mathscr{C}_n$ (resp. $\partial^+ \mathscr{C}_n$) in $\mathscr{C}_n$ remains bounded as $n \to \infty$. Hence, lifting to $\CH(\Lambda_n)$, we find that $w_\partial(\Lambda_n)$ is bounded as $n \to \infty$.
 
 On the other hand, $\mathscr{C}$ is not compact, but it has compact boundary $\partial^+ \mathscr{C}\cup\partial^- \mathscr{C}$. So there are points in $\mathscr{C}$, far out in the cusp, achieving arbitrary distance to both surfaces $\partial^+ \mathscr{C}$ and $\partial^- \mathscr{C}$. Hence there are points $x_n \in \mathscr{C}_n$ so that $d_{\mathscr{C}_n}(x_n, \partial^+ \mathscr{C}_n), d_{\mathscr{C}_n}(x_n, \partial^- \mathscr{C}_n) \to \infty$. Lifting to $\CH(\Lambda_n)$, we observe that $w(\Lambda_n) \to \infty$.
\end{proof}

Theorem \ref{tm:small} suggests that it could be possible that boundary width is equal to width for Jordan curves (quasicircles) of width bounded above by a small constant. 

\bibliographystyle{amsalpha} 
\bibliography{adsbib}

\end{document}